\documentclass[leqno]{amsart}
\usepackage{amsmath}
\usepackage{amssymb}
\usepackage{amsthm}
\usepackage{enumerate}
\usepackage[mathscr]{eucal}
\theoremstyle{plain}
\newtheorem{theorem}{Theorem}[section]
\newtheorem{prop}[theorem]{Proposition}
\newtheorem{cor}{Corollary}[theorem]
\newtheorem{lemma}{Lemma}[section]

\theoremstyle{definition}
\newtheorem{definition}{Definition}[section]
\newtheorem{remark}{Remark}[section]

\newtheorem{example}{Example}[theorem]

\usepackage[pagewise]{lineno}
\begin{document}
\title[Orthogonality and Numerical radius inequalities  ]{Orthogonality and Numerical radius inequalities of operator matrices}
\author[Arpita Mal, Kallol Paul and Jeet Sen]{Arpita Mal, Kallol Paul and Jeet Sen}

\newcommand{\acr}{\newline\indent}  

\address[Mal]{Department of Mathematics\\ Jadavpur University\\ Kolkata 700032\\ West Bengal\\ INDIA}
\email{arpitamalju@gmail.com}

\address[Paul]{Department of Mathematics\\ Jadavpur University\\ Kolkata 700032\\ West Bengal\\ INDIA}
\email{kalloldada@gmail.com}

\address[Sen]{Department of Mathematics\\ Jadavpur University\\ Kolkata 700032\\ West Bengal\\ INDIA}
	\email{senet.jeet@gmail.com}

\thanks{The research of  Arpita Mal is supported by UGC, Govt. of India.   The research of Jeet Sen is supported by CSIR, Govt. of India.} 


\begin{abstract}
 We completely characterize Birkhoff-James orthogonality with respect to numerical radius norm in the space of bounded linear operators on a complex Hilbert space.  As applications of the  results obtained, we estimate lower bounds of numerical radius for $n\times n$ operator matrices, which improve on and generalize existing lower bounds. We also obtain a better lower bound of numerical radius for an upper triangular operator matrix. 
\end{abstract}

\maketitle

\section{Introduction.} 
 The usual notion of orthogonality on an inner product space has been generalized on a Banach space by several mathematicians in various ways because of its importance in the study of geometry of Banach space. Birkhoff-James orthogonality \cite{B,J} is one of the most important notion of orthogonality among all others. Numerical radius of a bounded linear operator on a complex Hilbert space has also been studied extensively over the years. The purpose of this paper is to explore the connection between Birkhoff-James orthogonality and numerical radius norm of  bounded linear operators on a complex Hilbert space. As an application of the results obtained, we develop some bounds of the numerical radius for a bounded linear operator, which improve on  existing lower bounds of numerical radius. Before proceeding further, we announce the notations and terminologies to be used throughout the paper.\\

Let $\mathbb{H}$ denote a Hilbert space over the field $K,$ where $K\in \{\mathbb{R},\mathbb{C}\}.$ Let $B_{\mathbb{H}}$ and $S_{\mathbb{H}}$ denote the unit ball and the unit sphere of $\mathbb{H}$ respectively, i.e., $B_{\mathbb{H}}=\{x\in \mathbb{H}:\|x\|\leq 1\}$ and $S_{\mathbb{H}}=\{x\in \mathbb{H}:\|x\|=1\}.$ Let $\mathbb{B}(\mathbb{H})$ and $\mathbb{K}(\mathbb{H})$ denote the space of all bounded and compact linear operators on $\mathbb{H}$ respectively. For $x,y\in \mathbb{H},~x\otimes y$ denotes the rank one operator in $\mathbb{B}(H),$ defined by $(x\otimes y)(z)=\langle z,y\rangle x$ for all $z\in \mathbb{H}.$ An operator $T\in \mathbb{B}(\mathbb{H})$ can be represented as $T=H+iK,$ where $H=\frac{1}{2}(T+T^*),$ the real part of $T$ and $K=\frac{1}{2i}(T-T^*),$ the imaginary part of $T.$ For $T\in \mathbb{B}(\mathbb{H}),$ numerical radius $w(T)$, Crawford number $c(T)$ and minimum modulus $m(T)$  of $T$ are defined respectively as follows:
 \[w(T)=\sup\{|\langle Tx,x\rangle|:\|x\|=1\},\]
\[c(T)=\inf\{|\langle Tx,x\rangle|:\|x\|=1\} ~\text{and}\]
\[m(T)=\inf\{\| Tx\|:\|x\|=1\}.\]
Observe that if $\mathbb{H}$ is a complex Hilbert space, then the numerical radius $w(.)$ defines a norm on $\mathbb{B}(\mathbb{H})$ which is equivalent to the operator norm. In fact, for any $T\in \mathbb{B}(\mathbb{H}),$ $\frac{1}{2}\|T\|\leq w(T)\leq \|T\|.$ This inequality is sharp. If $T^2=0,$ then $w(T)=\frac{1}{2}\|T\|$ and if $T$ is self adjoint, then $w(T)=\|T\|.$ If $\mathbb{H}$ is a real Hilbert space, then the numerical radius $w(.)$ is not necessarily a norm on $\mathbb{B}(\mathbb{H}),$ in fact, it is a pseudo-norm. For $T\in \mathbb{B}(\mathbb{H}),$ let us denote the set of all numerical radius attaining vectors by $M_{w(T)}$ and the set of all norm attaining vectors by $M_T,$ i.e.,  
\[M_{w(T)}=\{x\in S_{\mathbb{H}}:|\langle Tx,x \rangle |=w(T)\}~\text{and}~M_T=\{x\in S_{\mathbb{H}}:\|Tx\|=\|T\|\}.\]
In a Banach space $\mathbb{X},$ Birkhoff-James orthogonality is defined in the following way.\\
For $x,y\in \mathbb{X},~x$ is said to be Birkhoff-James orthogonal to $y,$ written as $x\perp_B y,$ if $\|x+\lambda y\|\geq \|x\|$ for all $\lambda \in K.$ Recently, many authors have studied orthogonality on $\mathbb{B}(\mathbb{H})$ with respect to different norms \cite{BS,G,MSP,PSG,SP}. Motivated by these, we study  ``numerical radius orthogonality" on $\mathbb{B}(\mathbb{H}).$ 
\begin{definition}
	For $T,A\in \mathbb{B}(\mathbb{H}),$ we say that $T$ is numerical radius orthogonal to $A,$ written as $``T\perp_w A",$ if $w(T+\lambda A)\geq w(T)$ for all $\lambda \in \mathbb{C}.$
\end{definition}
 Although numerical radius of operators is not a norm on real Hilbert space, we can define  ``numerical radius orthogonality" on $\mathbb{B}(\mathbb{H}),$ where $\mathbb{H}$ is a real Hilbert space, as follows:
\begin{definition}
	For $T,A\in \mathbb{B}(\mathbb{H}),$ we say that $T$ is numerical radius orthogonal to $A,$ written as $``T\perp_w A",$ if $w(T+\lambda A)\geq w(T)$ for all $\lambda \in \mathbb{R}.$
\end{definition}
  In section 2, we characterize numerical radius orthogonality for operators on complex as well as real Hilbert spaces. In section 3, using the results obtained in section 2, we estimate lower bounds of numerical radius for $n\times n$ operator matrices, which improve on and generalize existing lower bounds. Finally, we give numerical examples to show that the bounds obtained by us are better than the existing ones. 

\section{Numerical radius orthogonality}

We begin this section with an easy proposition which follows from the definition of numerical radius orthogonality of operators.
\begin{prop}
	Let $T,A\in \mathbb{B}(\mathbb{H}).$ Then the following conditions are equivalent.\\
	(i) $T\perp_w A$\\
	(ii) $T^*\perp_w A^*$\\
	(iii) $\alpha T\perp_w \beta A$ for all $\alpha,\beta \in K\setminus \{0\}. $
\end{prop}

In the next proposition, we obtain a connection between numerical radius orthogonality and Birkhoff-James orthogonality for self-adjoint and nilpotent operators on a complex Hilbert space.

\begin{prop}
	Let $\mathbb{H}$ be a complex Hilbert space and $T,A\in \mathbb{B}(\mathbb{H}).$ Then the following conditions hold:\\
	(i) If $T=T^*,$ then $T\perp_w A\Rightarrow T\perp_B A.$\\
	(ii) If $T^2=0,$ then $T\perp_BA\Rightarrow T\perp_w A.$ 
\end{prop}
\begin{proof}
	(i) If $T=T^*,$ then $w(T)=\|T\|.$ Now, $T\perp_wA\Rightarrow w(T+\lambda A)\geq w(T)$ for all $\lambda \in \mathbb{C}.$ Hence, $\|T+\lambda A\|\geq w(T+\lambda A)\geq w(T)=\|T\|$ for all $\lambda \in \mathbb{C}.$ Thus, $T\perp_B A.$\\
	(ii) If $ T^2=0,$ then $w(T)=\frac{1}{2}\|T\|.$ Let $T\perp_B A.$ Then $\|T+\lambda A\|\geq \|T\|$ for all $\lambda \in \mathbb{C}.$ Hence, $w(T+\lambda A)\geq \frac{1}{2}\|T+\lambda A\|\geq \frac{1}{2}\|T\|=w(T)$ for all $\lambda \in \mathbb{C}.$ Thus, $T\perp_w A.$ 
\end{proof}

\begin{remark}
In general, these two notions of orthogonality $``T\perp_wA"$ and $``T\perp_BA"$ are not equivalent. As for example, if we consider 
$
T=\begin{bmatrix}
0&1\\
0&0
\end{bmatrix},
A=\begin{bmatrix}
1&1\\
0&2
\end{bmatrix},
$
then $T\perp_w A$ but $T\not\perp_B A.$ Again, if we consider 
$
T=\begin{bmatrix}
1&0\\
i&1
\end{bmatrix},
A=\begin{bmatrix}
i &\frac{\sqrt{5}+1}{2}\\
0&0
\end{bmatrix},
$
then $T\perp_B A$ but $T\not\perp_w A.$
\end{remark}

In the following theorem, we prove the main result of this section, which  characterizes numerical radius orthogonality of bounded operators on complex Hilbert space.

\begin{theorem}\label{th-001}
	Let $\mathbb{H}$ be a complex Hilbert space and $T,A\in \mathbb{B}(\mathbb{H}).$ Then $T\perp_w A$ if and only if for each $\theta \in [0,2\pi),$ there exists a sequence $\{x_n^{\theta}\}$ in $ S_{\mathbb{H}}$ such that the following two conditions hold:\\
	(i) $\lim_{n\to \infty}|\langle Tx_n^{\theta},x_n^{\theta}\rangle|=w(T),$\\ 
	(ii) $\lim_{n\to\infty}Re\{e^{-i\theta}\langle Tx_n^{\theta},x_n^{\theta}\rangle \overline{\langle Ax_n^{\theta},x_n^{\theta}\rangle}\}\geq 0.$
\end{theorem}
\begin{proof}
	We first prove the sufficient part of the theorem. Let $\lambda \in \mathbb{C}.$ Then $\lambda =|\lambda|e^{i\theta}$ for some $\theta \in [0,2\pi).$ By hypothesis, there exists  a sequence $\{x_n^{\theta}\}$ in $ S_{\mathbb{H}}$ such that $(i)$ and $(ii)$ hold.  Therefore,
	\begin{eqnarray*}
		w(T+\lambda A)^2&\geq &\lim_{n\to\infty}|\langle Tx_n^{\theta}+\lambda Ax_n^{\theta},x_n^{\theta}\rangle|^2\\
		&=& \lim_{n\to \infty}[|\langle Tx_n^{\theta},x_n^{\theta}\rangle |^2+2|\lambda|Re\{e^{-i\theta} \langle Tx_n^{\theta},x_n^{\theta}\rangle \overline{\langle Ax_n^{\theta},x_n^{\theta}\rangle} \} +\\
		&&|\lambda|^2 |\langle Ax_n^{\theta},x_n^{\theta}\rangle |^2]	\\
		&\geq&\lim_{n\to\infty}|\langle Tx_n^{\theta},x_n^{\theta}\rangle |^2\\
		&=& w(T)^2.                
	\end{eqnarray*}
	Thus, for every $\lambda \in \mathbb{C},~w(T+\lambda A)\geq w(T).$ Hence, $T\perp_wA.$\\
	Now, we prove the necessary part of the theorem. Let $T\perp_wA.$ Then for every $\lambda \in \mathbb{C},~w(T+\lambda A)\geq w(T).$ Let $\theta \in [0,2\pi).$ Then $w(T+\frac{e^{i\theta}}{n} A)\geq w(T)>w(T)-\frac{1}{n^2}$ for all $n\in \mathbb{N}.$ Therefore, for each $n\in \mathbb{N},$ there exists $x_n^{\theta}\in S_{\mathbb{H}}$ such that
	 
	 \begin{eqnarray}
	 |\langle Tx_n^{\theta}+\frac{e^{i\theta}}{n} Ax_n^{\theta},x_n^{\theta}\rangle|>w(T)-\frac{1}{n^2}.
	 \end{eqnarray}
	 Hence, for all   $n\in \mathbb{N},$
	\begin{eqnarray*}
			(w(T)-\frac{1}{n^2})^2&<&  |\langle Tx_n^{\theta}+\frac{e^{i\theta}}{n} Ax_n^{\theta},x_n^{\theta}\rangle|^2\\
		\Rightarrow w(T)^2-\frac{2}{n^2}w(T)+\frac{1}{n^4}&<& |\langle Tx_n^{\theta},x_n^{\theta}\rangle |^2+\frac{1}{n^2} |\langle Ax_n^{\theta},x_n^{\theta}\rangle |^2  \\
		&& + \frac{2}{n}Re\{e^{-i\theta} \langle Tx_n^{\theta},x_n^{\theta}\rangle \overline{\langle Ax_n^{\theta},x_n^{\theta}\rangle} \}\\
		\Rightarrow \frac{n}{2} [w(T)^2-|\langle Tx_n^{\theta},x_n^{\theta}\rangle |^2] &<& \frac{1}{n}w(T)-\frac{1}{2n^3} +\frac{1}{2n} |\langle Ax_n^{\theta},x_n^{\theta}\rangle |^2\\
		&&+ Re\{e^{-i\theta} \langle Tx_n^{\theta},x_n^{\theta}\rangle \overline{\langle Ax_n^{\theta},x_n^{\theta}\rangle} \}\\
		\Rightarrow 0&<&\frac{1}{n}w(T)-\frac{1}{2n^3} +\frac{1}{2n}\|A\|^2 \\
		&&+Re\{e^{-i\theta} \langle Tx_n^{\theta},x_n^{\theta}\rangle \overline{\langle Ax_n^{\theta},x_n^{\theta}\rangle} \}.
	\end{eqnarray*}
    Now, since $\{\langle Tx_n^{\theta},x_n^{\theta}\rangle\}$ and $\{\langle Ax_n^{\theta},x_n^{\theta}\rangle\}$ are bounded sequences of complex numbers, if necessary, passing through a subsequence and taking limit $n\to\infty$ in the last inequality, we get $0\leq  \lim_{n\to\infty}Re\{e^{-i\theta} \langle Tx_n^{\theta},x_n^{\theta}\rangle \overline{\langle Ax_n^{\theta},x_n^{\theta}\rangle} \}.$ This proves $(ii).$\\
     Now, we prove $(i).$ Using $(1),$ we get $ |\langle Tx_n^{\theta},x_n^{\theta}\rangle|\geq|\langle Tx_n^{\theta}+\frac{e^{i\theta}}{n} Ax_n^{\theta},x_n^{\theta}\rangle|-\frac{1}{n}|\langle Ax_n^{\theta},x_n^{\theta}\rangle|\\
		>w(T)-\frac{1}{n^2}-\frac{1}{n}\|A\|.$ Once again, taking limit $n\to\infty,$ we get $\lim_{n\to\infty} |\langle Tx_n^{\theta},x_n^{\theta}\rangle|\geq w(T).$ Clearly, $w(T)\geq\lim_{n\to\infty} |\langle Tx_n^{\theta},x_n^{\theta}\rangle|.$ Thus, $(i)$ holds. This completes the proof of the theorem. 
    \end{proof}
    
   If we consider compact operators instead of bounded operators then we get the following theorem.

    \begin{theorem}\label{th-01}
    	Let $\mathbb{H}$ be a complex Hilbert space and $T,A\in \mathbb{K}(\mathbb{H}).$ Then $T\perp_w A$ if and only if for each $\theta \in [0,2\pi),$ there exists $x_{\theta}\in M_{w(T)}$ such that $$Re\{e^{-i\theta}\langle Tx_\theta,x_\theta\rangle \overline{\langle Ax_\theta,x_\theta\rangle}\}\geq 0.$$
    \end{theorem}
    
    \begin{proof}
    	By Theorem \ref{th-001}, $T\perp_wA$ if and only if  for each $\theta \in [0,2\pi),$ there exists a sequence $\{x_n^{\theta}\}$ in $ S_{\mathbb{H}}$ such that the following two conditions hold:\\
    	(i) $\lim_{n\to \infty}|\langle Tx_n^{\theta},x_n^{\theta}\rangle|=w(T),$\\ 
    	(ii) $\lim_{n\to\infty}Re\{e^{-i\theta}\langle Tx_n^{\theta},x_n^{\theta}\rangle \overline{\langle Ax_n^{\theta},x_n^{\theta}\rangle}\}\geq 0.$\\
	Since every Hilbert space is reflexive, $B_{\mathbb{H}}$ is weakly compact. So without loss of generality, we may assume that for each $\theta \in[0,2\pi),$ there exists some $x_\theta \in B_{\mathbb{H}}$ such that $\{x_n^{\theta}\}$ weakly converges to $x_\theta.$ Now, $T,A$ being compact, $\lim_{n\to\infty}Tx_n^{\theta}=Tx_{\theta}$ and $\lim_{n\to\infty}Ax_n^{\theta}=Ax_{\theta}.$ Thus, $\lim_{n\to\infty}\langle Tx_n^{\theta},x_n^{\theta}\rangle = \langle Tx_\theta,x_\theta\rangle$ and $\lim_{n\to\infty}\langle Ax_n^{\theta},x_n^{\theta}\rangle=\langle Ax_\theta,x_\theta\rangle$. Now, taking limit $n\to \infty$ in $(i)$ and $(ii)$ we obtain, $x_{\theta}\in M_{w(T)}$ and $ Re\{ e^{-i\theta}\langle Tx_\theta,x_\theta\rangle \overline{\langle Ax_\theta,x_\theta\rangle}\}\geq 0.$ This completes the proof of the theorem. 
\end{proof}

In the following two theorems, we state characterizations of numerical radius orthogonality for bounded and compact operators on real Hilbert space, the proofs of which follow from Theorem \ref{th-001} and Theorem \ref{th-01}.
\begin{theorem}\label{th-002}
	Let $\mathbb{H}$ be a real Hilbert space and $T,A\in \mathbb{B}(\mathbb{H}).$ Then $T\perp_w A$ if and only if there exist sequences $\{x_n\},\{y_n\}$ in $ S_{\mathbb{H}}$ such that the following conditions hold:\\
	(i) $w(T)=\lim_{n\to\infty}|\langle Tx_n,x_n\rangle|=\lim_{n\to\infty}|\langle Ty_n,y_n\rangle|,$\\
	 (ii) $\lim_{n\to\infty}\langle Tx_n,x_n\rangle \langle Ax_n,x_n\rangle \geq 0,$ \\ (iii) $\lim_{n\to\infty}\langle Ty_n,y_n\rangle \langle Ay_n,y_n\rangle \leq 0.$ 
\end{theorem}

Similarly, numerical radius orthogonality for compact operators in real Hilbert Hilbert space can be characterized in the following way.

\begin{theorem}\label{th-02}
	Let $\mathbb{H}$ be a real Hilbert space and $T,A\in \mathbb{K}(\mathbb{H}).$ Then $T\perp_w A$ if and only if there exist $x,y\in M_{w(T)}$ such that $\langle Tx,x\rangle \langle Ax,x\rangle \geq 0$ and $\langle Ty,y\rangle \langle Ay,y\rangle \leq 0.$ 
\end{theorem}

The property that $``T\perp_w A$ if and only if there exists $z\in M_{w(T)}$ such that $\langle Az,z \rangle=0"$ is itself a very interesting one. In the following two corollaries, we obtain  sufficient conditions for this to hold.

\begin{cor}\label{cor-01}
		Let $\mathbb{H}$ be a real Hilbert space, $T,A\in \mathbb{K}(\mathbb{H}),w(T)\neq 0$ and  $M_{w(T)}=D\cup(-D),$ where $D$ is a connected subset of $S_{\mathbb{H}}.$  Then $T\perp_w A$ if and only if there exists $z\in M_{w(T)}$ such that $ \langle Az,z\rangle = 0.$
\end{cor} 
\begin{proof}
	The sufficient part is obvious. We only prove the necessary part. Let $T\perp_wA.$ Then by Theorem \ref{th-02}, there exist $x,y\in M_{w(T)}$ such that  $\langle Tx,x\rangle \langle Ax,x\rangle \geq 0$ and $\langle Ty,y\rangle \langle Ay,y\rangle \leq 0.$ Without loss of generality, we may assume that $x,y\in D.$ Since the function $\phi:D\to\mathbb{R}$ defined by $\phi(x)=\langle Tx,x\rangle \langle Ax,x\rangle $ is continuous and $D$ is connected, $\phi(D)$ is connected. Again, $\phi(x)\geq 0$ and $\phi(y)\leq 0.$ Therefore, there exists $z\in D$ such that $\phi(z)=0,$ i.e., $\langle Tz,z\rangle \langle Az,z\rangle = 0\Rightarrow \langle Az,z\rangle = 0.$ This completes the proof.
\end{proof}

\begin{cor}
	Let $\mathbb{H}$ be a real Hilbert space and $T,A\in \mathbb{K}(\mathbb{H}),w(T)\neq 0.$ Suppose there exists $\lambda \neq 0$ such that $w(T+\lambda A)=w(T).$ Then  $T\perp_w A$ if and only if there exists $z\in M_{w(T)}$ such that $\langle Az,z\rangle = 0.$
\end{cor}
\begin{proof}
	The sufficient part is obvious. We only prove the necessary part. Let $T\perp_w A.$ Then by Theorem \ref{th-02}, there exist $x,y\in M_{w(T)}$ such that  $\langle Tx,x\rangle \langle Ax,x\rangle \geq 0$ and $\langle Ty,y\rangle \langle Ay,y\rangle \leq 0.$ Without loss of generality, assume that $\lambda >0.$ If possible, suppose that   $\langle Tx,x\rangle \langle Ax,x\rangle> 0.$ Then $w(T+\lambda A)^2\geq |\langle Tx+\lambda Ax,x\rangle|^2=\langle Tx,x\rangle^2+2\lambda\langle Tx,x\rangle \langle Ax,x\rangle +\langle Ax,x \rangle ^2>\langle Tx,x\rangle^2=w(T)^2,$ contradicting the hypothesis. Thus, $\langle Tx,x\rangle \langle Ax,x\rangle = 0\Rightarrow \langle Ax,x \rangle=0.$ Similarly, if $\lambda<0,$ then $\langle Ay,y \rangle =0.$ This completes the proof.
\end{proof}

In the next corollary, we obtain a characterization of numerical radius orthogonality for a special type of rank one operators.
\begin{cor}
	Let $\mathbb{H}$ be a real Hilbert space. Then for $x,y\in S_{\mathbb{H}},$ $x\otimes x\perp_w y\otimes y$ if and only if $\langle x,y \rangle =0.$ 
\end{cor}
\begin{proof}
	Let $x\otimes x\perp_w y\otimes y.$ Clearly, $M_{w(x\otimes x)}=\{\pm x\}.$ Thus, by Corollary \ref{cor-01}, $\langle (y\otimes y)x,x\rangle=0\Rightarrow \langle x,y\rangle ^2=0\Rightarrow \langle x, y \rangle =0.$ Conversely, let $\langle x, y \rangle =0.$ Then $ \langle (y\otimes y)x,x\rangle=0.$ Thus, by Corollary \ref{cor-01}, $x\otimes x\perp_w y\otimes y.$
\end{proof}

If $T\in \mathbb{K}(\mathbb{H}),$ where $\mathbb{H}$ is a real Hilbert space, then $w(T)=0$ does not imply that $T=0.$ In the following corollary, we characterize those $T$ for which $w(T)=0,$ in terms of numerical radius orthogonality.

\begin{cor}
	Let $\mathbb{H}$ be a real Hilbert space and $T\in \mathbb{K}(\mathbb{H}).$ Then $w(T)=0$ if and only if for every $x\in S_{\mathbb{H}},~x\otimes x\perp_wT.$
\end{cor} 
\begin{proof}
	The proof follows easily from Corollary \ref{cor-01} and the fact that $M_{w(x\otimes x)}=\{\pm x\},$ for each $x\in S_{\mathbb{H}}.$
\end{proof}

\section{Application:Lower bounds of the numerical radius for operators}
In this section, we apply the results obtained in Section 2 to find some lower bounds of numerical radius for operators. In \cite[Th. 3.7]{HKS}, the authors obtained a lower bound of numerical radius for $2\times 2$ operator matrix using certain pinching inequalities \cite[Page 107]{Bh}. Here, we give an alternative proof of \cite[Th. 3.7]{HKS} without using such inequality. We also obtain a lower bound of numerical radius for $n\times n$ operator matrix. To do so, we need the following lemma.

\begin{lemma}\label{lem-2x2}
	Suppose $A,B,C,D\in \mathbb{B}(\mathbb{H}),$  where $\mathbb{H}$ is a complex Hilbert space. Then 
	\[
	w\Big(\begin{bmatrix}
	A&B\\
	C&D
	\end{bmatrix}\Big)=
	w\Big(\begin{bmatrix}
	A&iB\\
	-iC&D
	\end{bmatrix}\Big)
	\]
\end{lemma}
\begin{proof}
The proof easily follows from the fact that $w(U^*TU)=w(T)$ for any unitary operator $U$ and considering $U=\begin{bmatrix} -iI&O\\O&I\end{bmatrix},$ $T=\begin{bmatrix}
	A&B\\
	C&D
	\end{bmatrix}.$ 

\end{proof}

\begin{theorem}\label{th-2x2}
Let $\mathbb{H}$ be a complex Hilbert space and let $A,B,C,D\in \mathbb{B}(\mathbb{H}).$ Let
\[
T=\begin{bmatrix}
A&B\\
C&D
\end{bmatrix}.
\]
Then $w(T)\geq\max\{w(A),w(D),\frac{1}{2}w(B+C),\frac{1}{2}w(B-C)\}.$ 	
\end{theorem}	
\begin{proof}
	Let $\{X_n\}$ be a sequence in $S_{\mathbb{H}}$ such that $|\langle AX_n,X_n\rangle|\to w(A).$ Then $w(T)\geq|\langle T(X_n,O),(X_n,O)\rangle|=|\langle(AX_n,CX_n),(X_n,O)\rangle| = |\langle AX_n,X_n\rangle|.$ Thus, $w(T)\geq\lim_{n\to\infty}|\langle AX_n,X_n\rangle|\Rightarrow w(T)\geq w(A).$ Similarly, $w(T)\geq w(D).$ Now, we show that $w(T)\geq \frac{1}{2}w(B+C).$ Consider 
	\[
	R=\begin{bmatrix}
	A&O\\
	O&D
	\end{bmatrix},
	S=\begin{bmatrix}
	O&B\\
	C&O
	\end{bmatrix}.
	\]	
 Let $\{X_n\}$ be a sequence in $S_{\mathbb{H}}$ such that $|\langle \frac{B+C}{2}X_n,X_n\rangle|\to w(\frac{B+C}{2}),$ $Z_{1n}=\frac{1}{\sqrt{2}}(X_n,X_n)$ and $Z_{2n}=\frac{1}{\sqrt{2}}(-X_n,X_n).$ Then
 \begin{eqnarray*}
 \langle SZ_{1n},Z_{1n}\rangle &=& \frac{1}{2}\Big[\langle BX_n,X_n\rangle+\langle CX_n,X_n\rangle\Big]\\
 &=&\langle \frac{B+C}{2}X_n,X_n\rangle.
 \end{eqnarray*}
Similarly, $\langle SZ_{2n},Z_{2n}\rangle=-\langle \frac{B+C}{2}X_n,X_n\rangle.$ Thus, 
$$\lim_{n\to \infty}|\langle SZ_{1n},Z_{1n}\rangle|=\lim_{n\to\infty}|\langle SZ_{2n},Z_{2n}\rangle|=\lim_{n\to\infty}|\langle \frac{B+C}{2}X_n,X_n\rangle|=w(\frac{B+C}{2}).$$ 
Clearly, $\langle RZ_{1n},Z_{1n}\rangle=\langle RZ_{2n},Z_{2n}\rangle=\frac{1}{2}[\langle AX_n,X_n\rangle+\langle DX_n,X_n\rangle].$ Since $\langle SZ_{1n},Z_{1n}\rangle\\
=-\langle SZ_{2n},Z_{2n}\rangle$ and $\langle RZ_{1n},Z_{1n}\rangle=\langle RZ_{2n},Z_{2n}\rangle,$ so without loss of generality, we may assume that for all $n\in \mathbb{N},$ either
 $$Re\{\langle RZ_{1n},Z_{1n}\rangle \overline{\langle SZ_{1n},Z_{1n}\rangle}\}\geq 0 ~\text{or}~ Re\{\langle RZ_{2n},Z_{2n}\rangle \overline{\langle SZ_{2n},Z_{2n}\rangle}\}\geq 0.$$ Without loss of generality, let $Re\{\langle RZ_{1n},Z_{1n}\rangle \overline{\langle SZ_{1n},Z_{1n}\rangle}\}\geq 0.$ Then
\begin{eqnarray*}
w(T)^2=w(R+S)^2&\geq&  |\langle RZ_{1n}+SZ_{1n},Z_{1n}\rangle|^2\\
&=& |\langle RZ_{1n},Z_{1n}\rangle|^2+2Re\{\langle RZ_{1n},Z_{1n}\rangle \overline{\langle SZ_{1n},Z_{1n}\rangle}\}\\
&&+|\langle SZ_{1n},Z_{1n}\rangle|^2\\
&\geq&|\langle SZ_{1n},Z_{1n}\rangle|^2\\
\Rightarrow w(T)^2&\geq&\lim_{n\to\infty}|\langle SZ_{1n},Z_{1n}\rangle|^2=w(\frac{B+C}{2})^2.
\end{eqnarray*}
Thus,  
\begin{eqnarray}
w(T)\geq w(\frac{B+C}{2}).
\end{eqnarray}
Now, considering
\[
T_1=\begin{bmatrix}
A&iB\\
-iC&D
\end{bmatrix}
\]
and replacing $B,C$ by $iB,-iC$ respectively in inequality $(2),$ we get, $w(T_1)\geq w(i\frac{B-C}{2})=w(\frac{B-C}{2}).$ Now, using Lemma \ref{lem-2x2}, we get $w(T)=w(T_1).$ Thus, $w(T)\geq\max\{w(A),w(D),\frac{1}{2}w(B+C),\frac{1}{2}w(B-C)\}.$ This completes the proof of the theorem.	
\end{proof}

 In the following theorem, we obtain a lower bound of the numerical radius for an $n\times n$ operator matrix. 
 \begin{theorem}\label{th-nXn}
 	Let $H_1,H_2,\ldots,H_n$ be complex Hilbert spaces and $\mathbb{H}=\oplus_{i=1}^n H_i.$ Let $A=(A_{ij}),$ where $A_{ij}\in \mathbb{B}(H_j,H_i)$ and  $\mathbb{B}(H_j,H_i)$ denote the space of all bounded linear operators from $H_j$ to $H_i.$ Then 
 	\[w(A)\geq \max\Big\{w(A_{kk}),w(T_i):1\leq k\leq n,1\leq i\leq n
 	\Big\},\] where $T_i=(t^i_{jk})_{n\times n}$ for each $i=1,2,\ldots,n$ and 
 	\[	t^i_{jk} =
 	\left\{
 	\begin{array}{ll}
 	O  & \mbox{if } j=i ~\mbox{or} ~k=i \\
 	A_{jk} & \mbox{otherwise, } 
 	\end{array}
 	\right.
 	\]
 	i.e., 
 	\[
 	T_i=\begin{bmatrix}
 	A_{11} &\ldots&A_{1(i-1)}&O  &A_{1(i+1)}&\ldots&A_{1n}\\
 	\vdots &\vdots &\ldots&\vdots &\vdots&  \vdots   & \ldots \\
 	A_{(i-1)1} &\ldots&A_{(i-1)(i-1)}&O  &A_{(i-1)(i+1)}&\ldots&A_{(i-1)n}\\
 	O &\ldots&O&O&O&\ldots&O \\
 	A_{(i+1)1} &\ldots&A_{(i+1)(i-1)}&O  &A_{(i+1)(i+1)}&\ldots&A_{(i+1)n}\\
 	\vdots &\ldots&\vdots &\vdots&  \vdots   & \ldots &\vdots\\
 	A_{n1} &\ldots&A_{n(i-1)}&O  &A_{n(i+1)}&\ldots&A_{nn}	
 	\end{bmatrix}
 	\]
 \end{theorem}
 \begin{proof}
 	Let $1\leq i,k\leq n.$ Clearly, $w(A)\geq w(A_{kk}),$ since if  $\{X_{kn}\}\subset S_{H_k}$ such that $ \lim_{n\to\infty}|\langle A_{kk}X_{kn},X_{kn}\rangle|= w(A_{kk}),$ then for $Z_n=(0,0,\ldots,X_{kn},0,\ldots,0)\in S_{\mathbb{H}},$ $\lim_{n\to\infty}|\langle AZ_n,Z_n\rangle|=\lim_{n\to\infty}|\langle A_{kk}X_{kn},X_{kn}\rangle|=w(A_{kk}).$ Therefore, $w(A)\geq w(A_{kk}).$ We only show that $w(A)\geq w(T_i),$ where $T_i$ is defined as in the statement of the theorem. Suppose $S_i=A-T_i.$ We first show that $T_i\perp_w S_i.$ Let $X=(X_1,X_2,\ldots,X_n)\in S_{\mathbb{H}}.$ Then
 	\begin{eqnarray*}
 		|\langle T_iX,X\rangle|&=& |\sum_{j=1,j\neq i}^{n} \sum_{k=1,k\neq i}^{n}\langle A_{jk}X_k,X_j\rangle|.
 	\end{eqnarray*}
 	Now, suppose $\{X_m\}\subset S_{\mathbb{H}}$ be such that $\lim_{m\to\infty} |\langle T_i X_m,X_m\rangle|={w(T_i)}.$ We claim that there exists a sequence $\{Z_m\}$ in $ S_{\mathbb{H}}$ such that $w(T_i)=\lim_{m\to\infty}|\langle T_iZ_m,Z_m\rangle|$ and $(Z_m)_i=0$ for each $m\in \mathbb{N}.$ Suppose that $\sum_{k=1,k\neq i}^n\|(X_m)_k\|^2<1.$ Let $\alpha=\frac{1}{\sqrt{\sum_{k=1,k\neq i}^n\|(X_m)_k\|^2}}$ and $Z_m=\alpha(X_{m1},\ldots,X_{m(i-1)},0,X_{m(i+1)},\ldots,X_{mn}).$ Then $Z_m\\
	\in S_{\mathbb{H}}$ and $\alpha >1.$ Clearly, $|\langle T_iZ_m,Z_m\rangle|=\alpha^2|\sum_{j=1,j\neq i}^{n} \sum_{k=1,k\neq i}^{n}\langle A_{jk}X_{mk},X_{mj}\rangle|>|\langle T_iX_m,X_m\rangle|.$ Thus, $$w(T_i)\geq\lim_{m\to\infty}|\langle T_iZ_m,Z_m\rangle|\geq \lim_{m\to\infty}|\langle T_iX_m,X_m\rangle|=w(T_i).$$  This proves our claim. Now,
 	\begin{eqnarray*}
 		\langle S_iZ_m,Z_m \rangle&=&\langle (0,\ldots,0,\underset{i^{th}~\text{position}}{\underline{\sum_{j=1,j\neq i}^nA_{ij}X_{mj}}},0,\ldots,0),\\
 		&& (X_{m1},\ldots,X_{m(i-1)},0, X_{m(i+1)},\ldots,X_{mn})\rangle\\
 		&=&0.
 	\end{eqnarray*}
 	Thus, for each $\theta \in [0,2\pi),~Re\{e^{-i\theta}\langle T_iZ_m,Z_m \rangle \overline{\langle S_iZ_m,Z_m \rangle}\}=0.$ Hence, by Theorem \ref{th-001}, $T_i\perp_wS_i.$ Therefore, $w(A)=w(T_i+S_i)\geq w(T_i).$ This completes the proof of the theorem.
 \end{proof}
 
 Now applying Theorem \ref{th-nXn}, we obtain another lower bound of numerical radius for $n\times n$ operator matrix. 
\begin{theorem}\label{th-nXn2}
Let $H_1,H_2,\ldots,H_n$ be complex Hilbert spaces and $\mathbb{H}=\oplus_{i=1}^n H_i.$ Let $A=(A_{ij}),$ where $A_{ij}\in \mathbb{B}(H_j,H_i).$ Then 
\[w(A)\geq \max\Big\{w(A_{kk}),
w\Big(\begin{bmatrix}
A_{ii}&A_{ij}\\
A_{ji}&A_{jj}
\end{bmatrix}
\Big):\text{$1\leq k\leq n,1\leq i<j\leq n$}
\Big\}.\]
\end{theorem}
\begin{proof}
 Observe that $w(T_i)=w(R_i),$ where $T_i$ is defined as in Theorem \ref{th-nXn} and 
 	\[
 R_i=\begin{bmatrix}
 A_{11} &\ldots&A_{1(i-1)} &A_{1(i+1)}&\ldots&A_{1n}\\
 \vdots &\ldots &\vdots &\vdots&  \ldots   & \vdots \\
 A_{(i-1)1} &\ldots&A_{(i-1)(i-1)} &A_{(i-1)(i+1)}&\ldots&A_{(i-1)n}\\
 A_{(i+1)1} &\ldots&A_{(i+1)(i-1)}  &A_{(i+1)(i+1)}&\ldots&A_{(i+1)n}\\
 \vdots &\ldots&\vdots &  \vdots   & \ldots &\vdots\\
 A_{n1} &\ldots&A_{n(i-1)} &A_{n(i+1)}&\ldots&A_{nn}	
 \end{bmatrix}.
 \]
 Applying Theorem \ref{th-nXn}, repeatedly on $R_i$ for each $1\leq i\leq n$ we get the required inequality.
\end{proof}

In Theorem \ref{th-nXn2}, if we assume that $H_1=H_2=\ldots=H_n,$ then we get the following corollary.
\begin{cor}\label{cor-nxn}
	Let $\mathbb{H}$ be a complex Hilbert space and  $A=(A_{ij}),$ where $A_{ij}\in \mathbb{B}(\mathbb{H}).$   Then
	\[w(A)\geq \max \Big\{w(A_{kk}),\frac{1}{2}w(A_{ij}+A_{ji}),\frac{1}{2}w(A_{ij}-A_{ji}):1\leq i,j,k\leq n\Big\}.\] 
\end{cor}
\begin{proof}
	We only have to show that for $1\leq i,j\leq n,$
	\[w(A)\geq \max \Big\{\frac{1}{2}w(A_{ij}+A_{ji}),\frac{1}{2}w(A_{ij}-A_{ji})\Big\}.\] Without loss of generality, let us assume that $i<j.$ Then from Theorem \ref{th-nXn2} and Theorem \ref{th-2x2}, we have 
	\[
	w(A)\geq
	w\Big(\begin{bmatrix}
	A_{ii}&A_{ij}\\
	A_{ji}&A_{jj}
	\end{bmatrix}
	\Big)\geq \max \Big\{\frac{1}{2}w(A_{ij}+A_{ji}),\frac{1}{2}w(A_{ij}-A_{ji})\Big\}.
	\]
	This completes the proof of the corollary.
\end{proof}

In the next theorem, we obtain a lower bound of numerical radius for upper triangular $n\times n$ operator matrix.

\begin{theorem}\label{th-operatormatrix}
	Let $H_1,H_2,\ldots,H_n$ be complex Hilbert spaces and $\mathbb{H}=\oplus_{i=1}^n H_i.$ 
	Let 
	\[
	A = \begin{bmatrix} 
	A_{11} & A_{12} &A_{13} &\ldots&A_{1n}  \\
	O & A_{22}   & A_{23} &\ldots  &A_{2n} \\
	\vdots  &\ddots&\ddots& \ddots&\vdots\\
	\vdots  & &\ddots& \ddots&A_{(n-1)n}\\
	O&\ldots&\ldots&O&A_{nn}
	\end{bmatrix},
	\]
	i.e., $A=(A_{ij})_{n\times n}\in \mathbb{B}(\mathbb{H}),$ where for $1\leq i,j\leq n,~A_{ij}\in \mathbb{B}(H_j,H_i)$ and $A_{ij}=O$ if $i>j.$ Then $w(A)\geq \max\{w(A_{kk}),\frac{\|A_{ij}\|}{2}:i,j,k\in\{1,2,\ldots,n \},i<j\}.$ 
\end{theorem}
\begin{proof}
	Suppose $i,j\in\{1,2,\ldots,n \},i<j.$ From Theorem \ref{th-nXn2}, we have 
	\[
	w(A)\geq w\Big(\begin{bmatrix}
	A_{ii}&A_{ij}\\
	O&A_{jj}
	\end{bmatrix}\Big).
	\]
	We only have to show that 
	\[ w\Big(\begin{bmatrix}
	A_{ii}&A_{ij}\\
	O&A_{jj}
	\end{bmatrix}\Big)\geq\frac{\|A_{ij}\|}{2}.\]
	 Let
	\[	T =\begin{bmatrix}
	O&A_{ij}\\
	O&O
	\end{bmatrix},
	S=\begin{bmatrix}
	A_{ii}&O\\
	O&A_{jj}
	\end{bmatrix}.
	\]
	 We claim that $T\perp_wS.$ Clearly, $w(T)=\frac{\|A_{ij}\|}{2}.$ Let $\{X_{mj}\}$ be a sequence in $S_{H_j}$ such that $\lim_{m\to \infty}\|A_{ij}X_{mj}\|=\|A_{ij}\|$ and
	\[Z_{1m}=\frac{1}{\sqrt{2}\|A_{ij}\|}(A_{ij}X_{mj},\|A_{ij}\|X_{mj}),~Z_{2m}=\frac{1}{\sqrt{2}\|A_{ij}\|}(-A_{ij}X_{mj},\|A_{ij}\|X_{mj}).\]
	Then $\lim_{m\to\infty}\langle TZ_{1m},Z_{1m}\rangle=\frac{\|A_{ij}\|}{2}$ and $\lim_{m\to\infty}\langle TZ_{2m},Z_{2m}\rangle=\frac{-\|A_{ij}\|}{2}.$ Clearly, $\langle SZ_{1m},Z_{1m}\rangle=\langle SZ_{2m},Z_{2m}\rangle.$ So without loss of generality, we may assume that for each $\theta \in [0,2\pi),$ either
	$$Re\{e^{-i\theta}\langle TZ_{1m},Z_{1m}\rangle \overline{\langle SZ_{1m},Z_{1m}\rangle}\}\geq 0~\text{or}~ Re\{e^{-i\theta}\langle TZ_{2m},Z_{2m}\rangle \overline{\langle SZ_{2m},Z_{2m}\rangle}\}\geq 0.$$
 Therefore, by Theorem \ref{th-001}, $T\perp_wS.$ Hence, $w(T+S)\geq w(T)=\frac{\|A_{ij}\|}{2}.$ This completes the proof of the theorem.
\end{proof}

\begin{remark}
	$(i)$ Following \cite[Th. 3.7]{HKS}, for $A,B,C\in \mathbb{B}(\mathbb{H})$ and 
	\[ T = \left[ \begin{array}{ccc}
	A & B  \\
	O & D  \end{array} \right],\]
	it is easy to see that $w(T)\geq \max\{w(A),w(D),\frac{1}{2}w(B)\}.$ But from Theorem \ref{th-operatormatrix}, we get $w(T)\geq \max\{w(A),w(D),\frac{\|B\|}{2}\},$ which is clearly greater than or equal to $\max\{w(A),w(D),\frac{1}{2}w(B)\}.$ Thus, in this special case, Theorem \ref{th-operatormatrix} improves on \cite[Th. 3.7]{HKS}.\\
	
	$(ii)$ In \cite[Cor. 3.3]{GW}, Gau and Wu proved that if $A$ is an $n\times n$ block shift operator, i.e., if 
	\[
	A = \begin{bmatrix} 
	O & A_1 &  \\
      & O   & A_2 &  & \\
	 &  &O&\ddots    \\
	 &  &&\ddots& \ddots\\
	&&&& O& A_{k-1}\\
	&&&&&O
	\end{bmatrix},
	\]
	where $A_j$ is an $n_j\times n_{j+1}$ complex matrix, then $w(A)\geq \max\{w(B),w(C)\},$ where $B$ and $C$ are as follows.
		\[
	B = \begin{bmatrix} 
     0 & m(A_1) &  \\
     & 0   & m(A_2) &  & \\
     &  &0&\ddots    \\
     &  &&\ddots& \ddots\\
     &&&& 0& m(A_{k-1})\\
     &&&&&0
	\end{bmatrix}
	\text{and}
	\]
	\[
	C = \begin{bmatrix} 
	0 & m(A_1^*) &  \\
	& 0   & m(A_2^*) &  & \\
	&  &0&\ddots    \\
	&  &&\ddots& \ddots\\
	&&&& 0& m(A_{k-1}^*)\\
	&&&&&0
	\end{bmatrix}.
	\]
Clearly, if we choose all ${A_j}'s$ non-zero and $m(A_j)=m(A_j^*)=0,$ then Theorem \ref{th-operatormatrix} gives a better bound than \cite[Cor. 3.3]{GW}. Even if we choose 	${A_j}'s$ such that $m(A_j)$ and $m(A_j^*)$ are non-zero, then also Theorem \ref{th-operatormatrix} may give better bound than \cite[Cor. 3.3]{GW}. For example, if we consider 
\[
A_1=\begin{bmatrix}
4&0\\
0&1
\end{bmatrix},
A_2=\begin{bmatrix}
6&0\\
0&2
\end{bmatrix}
\text{and}
~~A=\begin{bmatrix}
O&A_1&O\\
O&O&A_2\\
O&O&O	
\end{bmatrix},	
\]
then clearly, 
\[
B=C=\begin{bmatrix}
0&1&0\\
0&0&2\\
0&0&0
\end{bmatrix}.
\]
Now, using the inequality $w(B)^2\leq \frac{1}{2}\|BB^*+B^*B\|$ \cite[Th. 1]{K}, we get $w(B)=w(C)<3=\max\{\frac{\|A_1\|}{2},\frac{\|A_2\|}{2}\}.$	
\end{remark}

Using Theorem \ref{th-nXn}, we can obtain a lower bound of the numerical radius for $n\times n$ scalar matrix.

\begin{theorem}\label{th-matrix}
	Let $A=(a_{ij})_{n\times n}\in M_{n\times n}(K).$ Then $w(A)\geq \max\{w(T_i):1\leq i\leq n\},$ where $T_i=(t^i_{jk})_{n\times n}$ for each $i=1,2,\ldots,n$ and 
	\[	t^i_{jk} =
	\left\{
	\begin{array}{ll}
	0  & \mbox{if } j=i ~\mbox{or} ~k=i \\
	a_{jk} & \mbox{otherwise. } 
	\end{array}
	\right.
	\]
i.e., 
	\[
	T_i=\begin{bmatrix}
	a_{11} &\ldots&a_{1(i-1)}&0  &a_{1(i+1)}&\ldots&a_{1n}\\
	\vdots &\vdots &\ldots&\vdots &\vdots&  \vdots   & \ldots \\
	a_{(i-1)1} &\ldots&a_{(i-1)(i-1)}&0  &a_{(i-1)(i+1)}&\ldots&a_{(i-1)n}\\
	0 &\ldots&0&0&0&\ldots&0 \\
	a_{(i+1)1} &\ldots&a_{(i+1)(i-1)}&0  &a_{(i+1)(i+1)}&\ldots&a_{(i+1)n}\\
		\vdots &\ldots&\vdots &\vdots&  \vdots   & \ldots &\vdots\\
	a_{n1} &\ldots&a_{n(i-1)}&0  &a_{n(i+1)}&\ldots&a_{nn}	
	\end{bmatrix}
	\]
\end{theorem}

\begin{remark}
In \cite{GWu}, Gau and Wu proved that if $T=(a_{ij})$ is an $n\times n$ complex matrix and $B=(b_{ij}),$ where 
	\[	b_{ij} =
\left\{
\begin{array}{ll}
a_{ij} & \mbox{if } (i,j)\in \{(1,2),(2,3),\ldots,(n-1,n),(n,1)\} \\
0 & \mbox{otherwise, } 
\end{array}
\right.
\]
then $w(T)\geq w(B).$ Observe that if $T=(a_{ij})$ be a non-zero $n\times n$ complex matrix such that $a_{ij}=0$ for all $(i,j)\in \{(1,2),(2,3),\ldots,(n-1,n),(n,1)\},$ then the matrix $B$ becomes zero matrix. Clearly, in this case Theorem \ref{th-matrix} gives a better bound than the one given in \cite{GWu}.
\end{remark}

Finally we give a concrete example to show that the lower bounds obtained by us are better than the existing lower bounds of numerical radius for a matrix. We first list some existing lower bounds. Let $T\in M_{n\times n}(\mathbb{C}).$ Then \\
$(1)$ \cite[Remark 2.2 (iii)]{KMY} $w(T)\geq \frac{1}{\sqrt{2}}\sqrt{\|H\|^2+\|K\|^2}.$\\
$(2) $ \cite[Remark 5]{OK}  $ w(T)\geq \frac{1}{2}\sqrt{\||T|^2+|T^*|^2\|+2c(T^2)}.$\\
$(3)$ \cite[Th. 3.3]{BBP} $w(T)\geq \sqrt{\|H\|^2+c^2(K)}.$\\
$(4)$  \cite[Th. 3.3]{BBP} $w(T)\geq \sqrt{\|K\|^2+c^2(H)}.$\\
$(5)$ \cite[Th. 4.1]{HKS} $w(T)\geq \frac{\|T\|}{2}+\frac{|\|H\|-\|K\||}{2}.$\\
$(6)$ \cite[Th. 4.2]{HKS} $w(T)\geq \frac{\|T\|}{2}+\frac{|\|H\|-\frac{\|T\|}{2}|}{4}+\frac{|\|K\|-\frac{\|T\|}{2}|}{4}.$

\begin{example}
	Let 
	\[
	T=\begin{bmatrix}
	2.6i & 4i & 0\\
	0   & 2.5i & 0\\
	0  &  0  & 1+i
	\end{bmatrix}.
	\]
	Then the lower bounds of numerical radius for $T,$ estimated by different mathematicians, mentioned in $1-6,$ are presented in the following table.

		\begin{center}
		\begin{tabular}{|l|c|c|c}
			\hline
			1 & Kittaneh, Moslehian and Yamazaki \cite[Remark 2.2]{KMY} & $2.783$\\
			\hline
		    2 & Abu-Omar and Kittaneh \cite[Remark 5]{OK}  & $3.654$ (approx)\\
		    \hline
			3 & Bhunia, Bag and Paul \cite[Th. 3.3]{BBP}& 2.236\\
			\hline
			4 & Bhunia, Bag and Paul \cite[Th. 3.3]{BBP}& 3.391\\
			\hline
			5 & Hirzallah, Kittaneh and Shebrawi \cite[Th. 4.1]{HKS}& 3.316\\ 
			\hline
			6 & Hirzallah, Kittaneh and Shebrawi \cite[Th. 4.2]{HKS}& 2.968\\ 
			\hline
		\end{tabular}
		\end{center}
But from Theorem \ref{th-matrix}, we get $w(T)\geq 4.55,$ which is better than all the estimations mentioned above. 

\end{example}

\end{document}